\documentclass[11pt]{article}
\usepackage{amsfonts,amsmath,amsthm}
\usepackage{verbatim}
\newtheorem{theo}{Theorem}[section]
\newtheorem{lem}[theo]{Lemma}
\newtheorem{cor}[theo]{Corollary}

\newtheorem{prop}[theo]{Proposition}

\newcommand{\mysection}[1]{\section{#1} \setcounter{equation}{0}}
\renewcommand{\proof}{{\sc Proof.} \quad}
\newcommand{\proofc}{{\sc Proof }}
\newcommand{\be}{\begin{equation} \label}
\newcommand{\ee}{\end{equation}}
\newcommand{\bea}{\begin{eqnarray}\label}
\newcommand{\eea}{\end{eqnarray}}
\newcommand{\bas}{\begin{eqnarray*}}
\newcommand{\eas}{\end{eqnarray*}}
\newcommand{\bit}{\begin{itemize}}
\newcommand{\eit}{\end{itemize}}
\renewcommand{\qed}{{}\hfill //// \\[15pt]}
\newcommand{\nn}{\nonumber}
\newcommand{\R}{\mathbb{R}}
\newcommand{\N}{\mathbb{N}}
\newcommand{\E}{\mathcal{E}}

\newcommand{\eps}{\varepsilon}

\newcommand{\abs}{\\[5pt]}

\newcommand{\ou}{\overline{u}}
\newcommand{\uu}{\underline{u}}
\newcommand{\ov}{\overline{v}}
\newcommand{\uv}{\underline{v}}
\newcommand{\oalpha}{\overline{\alpha}}
\newcommand{\ualpha}{\underline{\alpha}}
\newcommand{\hu}{\widehat{u}}

\newcommand{\ope}{{\cal E}}
\newcommand{\bx}{\overline{x}}
\newcommand{\hxz}{\widehat{x}_0}
\newcommand{\ba}{\begin{align}}
\newcommand{\ea}{\end{align}}
\begin{document}
\title{ A  singular differential equation stemming from an optimal control problem in financial economics}
\author{
Pavol Brunovsk\'y\footnote{{\bf Corresponding author:} brunovsky@fmph.uniba.sk; tel:+421 265425741; fax:+421 265412305}\\ {\small
Department of Applied Mathematics and Statistics, Comenius
University Bratislava,}\\ {\small 84248 Bratislava, Slovakia,}
\and
 Ale\v s \v Cern\'y\footnote{ales.cerny.1@city.ac.uk},\\
 {\small Cass Business School, City University London,}
\\\small{106 Bunhill Row, London EC1Y 8TZ, UK}
\and
Michael Winkler\footnote{michael.winkler@uni-due.de}\\
{\small Institut f\"ur Mathematik, Universit\"at Paderborn,}\\
{\small 33098 Paderborn, Germany} }
\date{}
\maketitle
\begin{abstract}
\noindent
  We consider the ordinary differential equation
  \bas
    x^2 u''=axu'+bu-c(u'-1)^2, \qquad x\in (0,x_0),
  \eas
  with $a\in\R, b\in\R$, $c>0$  and the singular initial condition $u(0)=0$,
  which in financial economics describes optimal disposal of an asset in a market with liquidity effects.
  It is shown in the paper that if $a+b < 0$  then no continuous solutions exist, whereas if $a+b>0$
  then there are infinitely many continuous solutions with indistinguishable asymptotics near 0. Moreover, it is proved that
  in the latter case there is precisely one solution $u$ corresponding to the choice $x_0=\infty$ which is such
  that $0 \le u(x)\le x$ for all $x>0$, and that this solution is strictly increasing and concave.\abs

\end{abstract}

\section*{Introduction}
The paper is concerned with solutions of the problem
\begin{align}
 x^2 u''&=axu'+bu-c(u'-1)^2, \ x > 0, \label{0}\\
    u(0)&=0, \label{0.1}
    \end{align}
where $a$ and $b$ are real numbers and $c>0$.  By a solution of
(\ref{0}) in $[0,x_0]$ we mean a function $u\in C^0([0,x_0]) \cap
C^2((0,x_0))$ which satisfies (\ref{0}) for $x>0$.

Equation \eqref{0} arises in the study of a specific stochastic optimization similar to the classical LQ problem. The equation is singular at $x=0$ which in itself is not particularly noteworthy, since stochastic LQ problems with geometric Brownian state variable invariably give rise to nonlinear singular ODEs/PDEs of the type seen in \eqref{0} and in \eqref{0.2} below, see for example \cite{wirl}. Our problem derives its rich structure from the fact that the initial condition \eqref{0.1}, too, refers to the singular point $x=0$. This, as we demonstrate below, poses certain technical obstacles in establishing existence and, more importantly, gives rise to infinitely many solutions with indistinguishable asymptotics near zero (Corollary \ref{nu}).

As was already highlighted, the ODE
\eqref{0}, \eqref{0.1} is not artificial, rather it stems from a well-defined optimization problem in financial economics. Specifically, it is obtained from the PDE
\be{0.2} \left\{\begin{array}{l} \frac12y^2\sigma^2w_{yy}+\lambda
yw_y + r^*zw_z-\rho w + \frac{(y-w_z)^2}{4\eta}=0,\\ w(y,0)=0,
\end{array} \right.
\ee
using the scaling $$w(y,z)=\frac{y^2}\eta u(x), \ x =\eta \frac zy
.$$

The PDE \eqref{0.2} in turn represents the Hamilton-Jacobi-Bellman
equation for the optimal value function $w$ of the following
dynamic optimization problem:
\be{0.2b}
 w(y(0),z(0))=\hbox{max } E\left(\int_0^{T(z=0)}e^{-\rho s}f(y(s),z(s))(y(s)-\eta f(y(s),z(s)))ds \right),
\ee
subject to
\begin{align}
 dy(t)&=\lambda y(t)dt + \sigma y(t)dB(t), \label{0.4}\\
dz(t)&=(r^*z(t)-f(y(t),z(t)))dt, \label{0.2c}
 \end{align}
over the controls $f:\mathbb R_+\times \mathbb R_+\mapsto \mathbb R$, $T(z=0)$ being the
first arrival time at $z=0$ and $B(t)$ the standard Wiener
process.

Optimization \eqref{0.2b} with dynamics given by (\ref{0.4}, \ref{0.2c}) models optimal liquidation of a large
quantity of an asset whose market price is adversely affected by
its ongoing sale. In this context $z(0)$ represents the quantity of the asset yet to be sold, $y(0)$ is the prevailing price
and $w$ captures the expected revenue of an optimal sale of quantity $z$ conditional on the current price being $y$.
For more details we refer the reader to \cite{cerny}.

The problem of
existence is the first principal subject of this
paper. It is shown that for $a+b> 0$ the problem \eqref{0}, \eqref{0.1} has a continuum of local solutions and at least one global solution bounded between 0 and $x$ (Section
1), while for $a+b<0$ no solutions exist (Section 2).

If one admits the possibility that there are multiple solutions to \eqref{0}, \eqref{0.1},
one immediately has to deal with the additional challenge of identifying ``the'' right solution relevant to the associated optimization problem.
The economic nature of the optimization (\ref{0.2b}-\ref{0.2c}) strongly suggests that the relevant solution of \eqref{0}, \eqref{0.1} should be increasing
(larger amount of asset means larger revenue) but concave
(decreasing returns to scale, since larger volume of sales has greater adverse effect on the sale price of the asset). However,
there is no indication in the form of equation \eqref{0} that a solution with these properties should exist in the first place.
In Section 3 we thus analyze monotonicity and convexity properties of a
solution bounded between $0$ and $x$ (Proposition \ref{uprime}), the upper bound corresponding to an immediate sale of the entire stock of the asset without any adverse price effect.

In Section 4 we address the question of global uniqueness.
We show that there is exactly
one solution on $\mathbb{R}_+$ which remains bounded between 0 and $x$, and this
solution is necessarily increasing and concave (Proposition \ref{uniqueness2}).
Finally, in Section 5 we examine finer aspects of local non-uniqueness.

In the current paper we focus on the intricacies of the initial
value problem \eqref{0}, \eqref{0.1}. The implication of the
results for the underlying optimal control problem is a delicate
issue left to further research.

A paper similar in spirit to ours is \cite{Cidetal}. It studies
a specific second order equation with a singularity at 0,
arising in the theory of general relativity.
As in our case, the existence of infinitely many solutions
is established by the method of upper and lower solutions and then
the properties of the set of solutions are studied.

As far as the local existence is concerned, Liang in  \cite{liang} carried out a systematic study of second
order singular initial value problems of the form \be{liang} u'' =
\frac{1}{x}F(x,u,u'), \ee where $F$ is a continuous function and
the initial conditions satisfy $F(0,u(0),u'(0))=0$. The key
quantity in this study is $\gamma:=\frac{\partial}{\partial
u'}F(0,u(0),u'(0))$. It is shown that for $\gamma < 0$ local
uniqueness holds, while for $\gamma>0$ solutions become unique
only after the asymptotics of $u'$ have been fixed to the order
$x^\gamma$ near $x=0$. The case $\gamma=0$ is not treated. Each
solution has an asymptotic expansion in powers of $x$ and
$x^\gamma$ (provided $\gamma$ is not an integer), and asymptotic
expansion of $u^{(n)}$ is obtained by differentiating $n$-times
the asymptotic expansion for $u$.

In contrast, we study a specific singular IVP from a wider class
\be{us} u'' = \frac{1}{x^\alpha}F(x,u,u'), \ee with $\alpha = 2$,
$u(0)=0$, $u'(0)=1$ and $F(x,u,u'):=axu'+bu-c(u'-1)^2$. Like in
\cite{liang} our ODE arises from a self-similar solution of a PDE.
However, we deal with a borderline case where
$\frac{\partial}{\partial u'}F(0,u(0),u'(0))=0$. As a result,
standard blow-up techniques are not productive and we have to
resort to the method of sub-supersolutions.

Finally, we remark that it is not uncommon for HJB equations
associated with stochastic optimization to exhibit multiple
solutions. The meaningful solution then has to be selected by
employing additional criteria. In the case of linear-quadratic
problems the relevant solution is identified as the
maximal/minimal one. In other cases the optimal solution can be
singled out as the unique viscosity solution of the HJB equation,
cf. \cite{Barles}. In our case these criteria do not seem to be
helpful. Rather, the significant solution is uniquely determined
by its global monotonicity and concavity properties.
\mysection{Existence for $a+b\ge 0$}\label{sect_exist}
In essence, existence will be proved similarly as in
\cite{Cidetal}. That is,  ordered pairs of  sub- and a
supersolutions of  (\ref{0}) will be found, and an application of
a standard existence result for second order boundary-value
problems will provide solutions lying in between, cf.
\cite{decoster}. As in \cite{Cidetal}, due to the singularity in
the ODE (\ref{0}), an approximation procedure will be involved in
the proof. However, compared to \cite{Cidetal}, the presence of
$u'$ in the equation will require additional arguments. We isolate
technical arguments in the following propositions.
\begin{prop}\label{prop1} Let $0<x_1<x_2$ and suppose that there
exist $\uu,\ou\in C^2([x_1,x_2])$ such that
\begin{align}
\uu&\le\ou \mbox{ in }  [x_1,x_2];\\ \ope \uu &< 0 \mbox{ in }
[x_1,x_2];\\ \ope \ou &> 0  \mbox{ in }  [x_1,x_2];
\end{align}
 the operator $\ope$ being defined according to \be{8.2}
    \ope u := -x^2 u'' + axu' + bu - c(u'-1)^2
\ee for functions $u$ which belong to $C^2([x_1,x_2])$. Then for
each $u_1\in[\uu(x_1),\ou(x_1)], \ u_2\in[\uu(x_2),\ou(x_2)]$
there exists a solution $u\in C^2([x_1,x_2]) $ to \eqref{0} in
$[x_1,x_2]$ satisfying $\uu\le u\le\ou, \ u(x_1)=u_1, \
u(x_2)=u_2$.
\end{prop}
\begin{proof}
Rewrite equation \eqref{0} as $$u''=f(x,u,u')$$ with $$f(x,u,u')=
x^{-1}au' + x^{-2}bu - x^{-2}c(u'-1)^2.$$ For $x_1\le x\le x_2$
and $\uu(x)\le u\le  \uu(x)$,  $f$ satisfies the Bernstein
condition \cite{bernstein} $$|f(x,u,u')|\le A+Bu'^2$$ for suitable
$A,B>0$. Therefore, the result follows from  Nagumo \cite{nagumo},
Satz 2, cf. also \cite{decoster}, Theorem II-1.3 for a more recent
reference.
\end{proof}

{\bf Remark.} Recall that regularity of a differential equation is
inherited by its solutions (cf. \cite{hartman}, Chapter V,
Corollary 4.1). In particular, since the expression for $u''$ is
$C^\infty$ in $x,u,u'$ for $x>0$, any solution $u$ of \eqref{0} in
$[x_1,x_2]$ with $0<x_1$ is in $C^\infty([x_1,x_2])$.

\begin{prop}\label{prop1a}
\noindent (i) Let $x_0\in (0,\infty)$. Suppose that there exist
$\uu,\ou\in C^0[0,x_0]\cap C^2(0,x_0)$ satisfying
\begin{equation}\label{vnule} \uu(0)=\ou(0)=0
\end{equation}
in addition to (1.1)-(1.3) with $x_1=0, x_2=x_0$. Then, for each
$u_0\in [\uu(x_0), \ou(x_0)]$ there exists a solution of
(0.1),(0.2) in $[0,x_0]$ such that $\uu \le u \le \ou$ in
$(0,x_0)$ and $u(x_0)=u_0$.

\noindent (ii) Let $\uu, \ou$ satisfy (1.1)-(1.3) for $x_1=0$ and
$x_2=\infty$ as well as \eqref{vnule}. Then, there exists a
solution of (0.1), (0.2) in $[0,\infty)$ such that $\uu \le u \le \ou$.
\end{prop}

\begin{proof} (i) By Proposition \ref{prop1}, for each $\eps \in (0,x_0)$ and each
$u_0\in [\uu(x_0), \ou(x_0)]$ there exists a solution
  $u_\eps \in C^2([\eps,x_0])$ of
  \be{9.12}
    \left\{ \begin{array}{l}
    \ope u_\eps =0 \qquad \mbox{in } [\eps,x_0], \\[1mm]
    u_\eps(\eps)=\uu(\eps), \quad u_\eps(x_0)=u_0,
    \end{array} \right.
  \ee
  which satisfies
  \be{9.13}
    \uu(x) \le u_\eps(x) \le \ou(x) \qquad \mbox{for all } x\in (\eps,x_0).
  \ee
Let now $\eps_n\searrow 0$ for $n\to\infty$. For fixed $n$, the functions
$u_{\eps_k}$ with  $k\ge n$ are uniformly bounded on
$[\eps_n,x_0]$ . By \cite{bernstein} (cf. also \cite{decoster},
I.4.3 page 45), the same holds for their derivatives
$u'_{\eps_k}$. Therefore, on $[\eps_n,x_0]$, $u_{\eps_k}$ are
equicontinuous and, moreover, from \eqref{0} it follows that
$u''_{\eps_k}$, $k\ge n$ are uniformly bounded on $[\eps_n,x_0]$.
Thus $u'_{\eps_k}$ are equicontinuous and, in turn, because od
\eqref{0}, $u''_{\eps_k}$ are equicontinous as well on
$[\eps_n,x_0]$. Therefore, one can pick a subsequence
$u_{\eps_{k_j}}$ which converges $C^2$ uniformly to a $C^2$
function $u^n$ satisfying \eqref{0} on $[\eps_n,x_0]$ together
with  $u^n(x_0)=u_0$ and $\uu(x)\le u^n(x)\le \ou(x)$ . By
standard diagonal selection we can pick a subsequence from the
sequence $u_{\eps_{k_j}}$ which converges pointwise in $[0, x_0]$
and uniformly in $[\eps,x_0]$ for each $0<\eps\le x_0$ to a
function $u\in C^0[0,x_0]\cap C^2(0,x_0]$ and satisfying the
requriements of item (i) of the Proposition.
\newline (ii) By (i), for each  $\eps$ we have a solution of
  (0.1), (0.2) such  that $u_\eps (0)=0$ and $\uu \le u \le \ou$
  in $[0, 1/\eps]$. Applying for $\eps>0$ the same extraction idea as in (i) we
  obtain the claimed solution in $[0,\infty)$.
\qed
\end{proof}

%
\begin{prop}\label{basicex}
i) For $a+b>0$ and any $x_0>0$ there is a continuum of solutions to (\ref{0}), (\ref{0.1}) on $[0,x_0]$ such that $0\le u \le x$.

ii) For $a+b\ge 0$ there is at least one solution of (\ref{0}), (\ref{0.1}) on $[0,\infty)$ such that $0\le u \le x$.
\end{prop}
\begin{proof}
For $a+b>0$ it is readily checked that $\uu(x)\equiv 0$ is a
subsolution and $\ou(x)=x$ is a supersolution in $[0,\infty)$. The
claim thus follows from Proposition \ref{prop1a}. For $a+b=0$,
$u(x)=x$ is a global solution.\qed
\end{proof}

The problem \eqref{0}, \eqref{0.1} can for $a+b>0$ be formally solved by a
power series. We let \be{7.1}
    k_{0} :=1, \quad
      \quad
    k_{1} :=-\frac{2}{3}\sqrt{\frac{a+b}{c}},
\ee
and inductively define
\be{7.2}
    f_{n}(x):=\sum_{i=0}^{n}k_{i}x^{1+i/2},
\ee
where
\be{7.3}
    k_{n+1} :=\lim_{x\rightarrow 0_{+}}\frac{2\mathcal{E}f_{n}}{3ck_{1}(n+3)x^{(n+2)/2}}
\ee
for $n\ge 1$.

\begin{lem}\label{lem8}
Let $a+b>0$. Then the coefficients $\{k_{i}\}_{i=0}^n$ are well-defined and $\mathcal{E}f_{n}=O(x^{(n+2)/2})$
as $x\searrow 0$
for all $n\in \mathbb{N}$.
\end{lem}

\proof
The statement clearly holds for $n=1$. Arguing by induction, we suppose that it is valid for some $n\ge 1$.
Then
\begin{align}
\mathcal{E}f_{n+1} &=\mathcal{E}f_{n}-2c(f_{n}^{\prime}-1)(1+(n+1)/2)k_{n+1}x^{(n+1)/2}+O(x^{(n+3)/2})\notag \\
&=\mathcal{E}f_{n}-\frac{3}{2}ck_{1}(n+3)k_{n+1}x^{(n+2)/2}+O(x^{(n+3)/2})
\qquad \mbox{as } x\searrow 0. \label{8.1}
\end{align}
Since $\E f_{n}$ is a polynomial in powers of $\sqrt{x}$ and $\E
f_{n}=O(x^{(n+2)/2})$ it follows that $k_{n+1}$ is well defined
and that $\E
f_{n}-(3/2)ck_{1}(n+3)k_{n+1}x^{(n+2)/2}=O(x^{(n+3)/2}).$ In view
of (\ref{8.1}) this implies that $\E f_{n+1}=O(x^{(n+3)/2})$ and
thus completes the proof. \qed

Easy calculations show that the coefficients $\{k_{n}\}_{n\geq 2}$
satisfy the
recursion%
\begin{align}
k_{2i} =&\frac{1}{6(i+1)k_{1}}\bigg[ 2\frac{k_{2i-1}}{c}\left( a+b+(i-\frac{%
1}{2})a-(i^{2}-\frac{1}{4})\right)    \notag \\
&\quad\quad\quad\quad\quad\quad\quad\quad\quad -\sum_{j=1}^{i-1}(3+j)(2+2i-j)k_{j+1}k_{2i-j}\bigg],  \label{r1}\\
k_{2i+1} =&\frac{1}{3(2i+3)k_{1}}\bigg[ 2\frac{k_{2i}}{c}\left(
a+b+ia-i(1+i)\right)  -\frac{1}{2}(3+i)^{2}k_{i+1}^{2}  \notag \\
&\quad\quad\quad\quad\quad\quad\quad\quad\quad -\sum_{j=1}^{i-1}(3+j)(3+2i-j)k_{j+1}k_{2i-j+1}\bigg]. \label{r2}
\end{align}%
>From here it is readily seen that the radius of convergence of the power
series \eqref{7.2} is nil when $a<\frac{3}{2}$ and $b\in (-a,\frac{3}{4}-\frac{3}{2%
}a],$ firstly by showing inductively $k_{i}>0$ for $i\geq 2$ and
subsequently neglecting all quadratic terms in $k_{i}$ in \eqref{r1}, \eqref{r2} and proving the easy
estimate $k_{n+1}/k_{n}\geq -2(n-1)/(3k_{1}c)$ for sufficiently large $n$.
Hence the power series $f_{n}$ does not define a solution directly via $%
\lim_{n\rightarrow \infty }f_{n}(x)$ outside $x=0$. We conjecture this
remains to be the case for arbitrary parameter values as long as $a+b>0$.

We will show later (Corollary \ref{nu}) that every local solution of (0.1), (0.2) with the property $u(x)\leq x$ satisfies
\[
u^{(k)}(x)=f^{(k)}_n(x)+o(x^{(n+3)/2-k}),
\]
for $k\in\{0,1\}$ and $n=1$. Whether this
is true for $n>1$ or $k>1$ remains an open question.
\mysection{Nonexistence for $a+b<0$}
In this second part we shall deduce Proposition \ref{prop4} below
which will exclude the existence of any continuous solution to
(\ref{0}) for any $x_0>0$ under the assumption $a+b<0$ which is
complementary to the hypothesis of Proposition \ref{basicex}.\\ To
this end we first prove that any supposedly existing continuous
solution must satisfy $u'(x)\to 1$ as $x\to 0$. This property can
formally easily be guessed upon tracing the possible solution
behavior near $x=0$.
\begin{lem}\label{u1}
Suppose that for some $x_0>0$, the function $u\in C^0([0,x_0]) \cap C^2((0,x_0))$ is a solution of
(\ref{0}), (\ref{0.1}). Then
\be{to 1}
    \lim_{x\searrow 0} u'(x)=1.
\ee
\end{lem}
\proof
  Letting $v:=u'-1$ we can rewrite (0.1) as
  \begin{eqnarray*}
     u'&=&v+1,\\ x^2 v'&=&ax(v+1)+bu-cv^2.
  \end{eqnarray*}
  Let $X(t):=-t^{-1}$ for $t<0$. Then $X'(t)=  t^{-2}$ and $X(t)\searrow 0$ as  $t\to-\infty$.
  We next introduce $U(t):=u(X(t))$ and $V(t):=v(X(t))$ for $t<0$. Then  the pair $(U,V)$ solves
  the  following system of differential equations
  \begin{eqnarray*}
    U'&=&  t^{-2}(V+1),\\ V'&=&-at^{-1}(V+1)+bU-cV^2.
  \end{eqnarray*}
  By assumption, we have $U(t)\to 0$ as $t\to-\infty$ and thus
  \be{Vprime}
    V'(t)=p(t)+q(t)V(t)-cV^2(t)
  \ee
  with
  \be{pq}
  p(t)\to 0 \mbox{ and }  q(t)\to 0 \mbox{ as } t\to -\infty.
  \ee
    We wish to show that {if $V(t)$ is defined for all $t\le -x_0^{-1}$} then $V(t)\to 0$ as $t\to-\infty$. The proof proceeds in several steps.\\
 i) Given $\eps>0$ there is $T<-x_0^{-1}$ such that
  $|p(t)|<\eps^2c/3$ and $|q(t)|<\eps c/3$ for all $t\le T$, by virtue of \eqref{pq}.
  \\
 ii) Consider $t_0\le T$. We claim that if $|V(t)|\ge\eps$ for
all $t\le t_0$ then \be{solution}
    V(t)\ge \frac{1}{V(t_0)^{-1}+\frac{c}{3}(t-t_0)} \mbox{ for all } t\le t_0.
  \ee
  while defined. To this end note that \eqref{Vprime} and i) yield
  \be{quadratic}
    V'(t)\le-\frac{c}{3}V^2(t) \qquad \mbox{ if } t\le T \mbox{ and } |V(t)|\ge \eps.
  \ee
  By the comparison theorem for ordinary differential equations we conclude that
  \bas
    V(t)\ge Y(t) \mbox{ for }t\le t_0,
  \eas
  where $Y$ solves the differential equation $Y'=-\frac c3 Y^2$ with $Y(t_0)=V(t_0)$.
  On solving for $Y$ we obtain \eqref{solution}.\\
 iii) Now we prove that there exists $t_1\le T$ such that
$V(t_1)> -\epsilon$. Suppose to the contrary that
$V(t)\le-\epsilon$ for all $t\le T$. Then, \eqref{solution} gives
$V(t)\ge-\eps/2$ for $t<t_0-\frac{6}{c\eps}$, yielding the desired
contradiction.

 Next we show that $V(t)> -\eps$ \emph{for all} $t\le t_1$.
Arguing by contradiction, suppose this is not the case. Then there
is $t_2$ such that $-\infty<t_2=\sup\{t\le t_1:V(t)\le -\eps\}<
t_1$. By continuity we have $V(t_2)= -\eps$. From
\eqref{quadratic} we obtain $V'(t_2)<0$ which is in conflict with
$V(t_2)=-\eps$ and $V(t)>-\eps$ for $t\in(t_2,t_1)$.

iv) Finally, we show that $V(t)\le\epsilon$ for all $t\le T$. If
not, there is $t_3\le T$ such that $V(t_3)> \epsilon$ and we have
$t_4:=\sup\{t\le t_3:V(t)\le \epsilon\}<t_3$. The same argument as
in part iii) shows that $t_4=-\infty$ and therefore $V(t)>\eps$
for all $t\le t_3$. From \eqref{solution} we now obtain
$V(t)\to\infty$ for $t\searrow t_3-\frac3cV(t_3)^{-1}$
 Therefore, $V(t)$ is not defined for some $t\le
-x_0^{-1}$ which is inconsistent with differentiability of $U$ in
$(-\infty,0)$.\abs

Since $\epsilon$ was arbitrary this completes the proof of the
lemma.\qed

It is now possible to rule out local existence of a continuous
solution of (\ref{0}), (\ref{0.1}) under the condition that $a+b$ be strictly
negative.
\begin{prop}\label{prop4}
  Suppose that $a+b<0$.
  Then for each $x_0>0$, the problem (\ref{0}), (\ref{0.1})
  does not possess any solution $u$ in $[0,x_0]$.
\end{prop}
\proof
  Suppose that such a solution exists for some $x_0>0$. Then from Lemma \ref{u1} we know that
  $u$ actually belongs $C^1([0,x_0])$ with $u'(0)=1$, and hence the functions $\varphi_1$ and $\varphi_2$ defined by
  \bas
    \varphi_1(x):=u'(x)-1, \quad x\in (0,x_0),
    \qquad \mbox{and} \qquad
    \varphi_2(x):=\frac{u(x)-x}{x}, \quad x\in (0,x_0),
  \eas
  satisfy $\varphi_1(x)\to 0$ and $\varphi_2(x)\to 0$ as $x\to 0$.
  Since $a+b<0$, we can thus find $\bx \in (0,x_0)$ such that
  \bas
    a+b + a\varphi_1(x)+b\varphi_2(x) \le \frac{a+b}{2} \qquad \mbox{for all } x\in (0,\bx).
  \eas
  Therefore, (\ref{0}) shows that
  \bas
    x^2 u''(x) &=& axu'(x) +bu(x) -c(u'(x)-1)^2 \\
    &\le& axu'(x) +bu(x) \\
    &=& ax(1+\varphi_1(x)) + bx(1+\varphi_2(x)) \\
    &=& \Big(a+b+a\varphi_1(x)+b\varphi_2(x)\Big) \cdot x \\
    &\le& -\delta x \qquad \mbox{for all } x\in (0,\bx)
  \eas
  holds with $\delta:=-\frac{a+b}{2}>0$. By integration we find that
  \bas
    u'(\bx) - u'(x) \le -\delta \ln \frac{\bx}{x} \qquad \mbox{for all } x\in (0,\bx).
  \eas
  This implies that $u'(x)\to + \infty$ as $x\to 0$ and thereby contradicts Lemma
  \ref{u1}.
\qed
\mysection{Monotonicity and concavity properties of solutions}
In this section we assume $a+b>0$ and we study monotonicity and convexity properties of
solutions to \eqref{0}, whose existence was established in Section 1.\abs
The following lemma is the key to establishing monotonicity, concavity, and ultimately also uniqueness in a certain restricted class of solutions.
\begin{lem}\label{one root}
  Consider a nonconstant function $y\in C^0([0,\infty)) \cap
  C^2((0,\infty))$ satisfying \be{basic}
  x^2y''(x) = f(x)y'(x)+g(x,y(x)),\ee for some continuous functions
  $f$ and $g$. Suppose there is a constant $y^*\in [-\infty,\infty]$
  such that for all $x>0$ one has $g(x,y)>0$ for $y>y^*$ and $g(x,y)<0$
  for $y<y^*$. Then there is at most one $x_0\in
  (0,\infty)$ such that $y'(x_0)=0$. If such $x_0$ exists then one,
  and only one, of the following two alternatives is possible:
  Either
  \bas
    \mbox{$y'(x)<0$ for $x<x_0$, $y'(x)>0$ for $x>x_0$, and $y(x)>y(x_0)>y^*$ for all $x\neq x_0$,}
  \eas
  or
  \bas
    \mbox{$y'(x)>0$ for $x<x_0$, $y'(x)<0$ for $x>x_0$, and $y(x)<y(x_0)<y^*$ for all $x\neq x_0$.}
  \eas
\end{lem}
\proof
We first note that because of continuity  of $g$ we have  $g(x,y^*)=0$
for all $x>0$ whenever $y^*$ is finite. By an ODE uniqueness argument,
$y(x_0)=y^*$ and  $y'(x_0)=0$ implies $y(x)\equiv y^*$.
Therefore, if $y(x)$ is not constant and $y'(x_0)=0$ then
$y(x_0)\ne y^*$.\\
Now suppose that $y(x_0)>y^*$. Then from \eqref{basic} it follows
that $y''(x_0)>0$, hence $y'(x)<0$ for $x<x_0$ sufficiently close to $x_0$.
Arguing by contradiction, let us suppose that there exists $0<x_1<x_0$ such that $y'(x_1)\ge 0$. Then
there is $x_2\in [x_1,x_0)$ such that
\be{flat} y'(x_2)=0, \ y'(x)<0 \hbox{ for }x_2<x<x_0, \ee
which implies $y''(x_2)<0$ and also $y(x)>y(x_0)>y^*$
for $x_2\le x <x_0$. On the other hand, \eqref{basic} together with $y(x_2)>y^*$ and $y'(x_2)=0$ entails that
$y''(x_2)>0$, yielding the desired contradiction.
Therefore, $y'(x)<0$ for all $x\in (0,x_0)$. The proof of
$y'(x)>0$ for $x>x_0$ follows the same lines. Finally, the case
$y(x)<y^*$ can be reduced to the case $y(x)>y^*$ by the
transformation $y\mapsto -y$, $y^*\mapsto -y^*$.
\qed
We now apply this to derive some monotonicity properties of solutions.
Here in order to abbreviate notation, we call a function $\phi:
[0,\infty)\to \mathbb R$ {\em eventually monotonic} if it is monotonic on $[x_0,\infty)$ for some $x_0\ge 0$.
\begin{lem}\label{monotony}
 Let $u$ be a nonconstant solution of (\ref{0}) on $(0,\infty)$.\\
  i) \ If $b\ne 0$ and $u$ is bounded and eventually monotonic, then $u(x)$ converges to the unique
  stationary solution $\hat u=c/b$ as $x\to\infty$. If $b=0$ and $u$ is eventually monotonic, then $u$ is unbounded.\\
  ii) \ If $b>0$ and $u\geq 0$, then one of the following alternatives occurs: Either
  \bas
    \mbox{$u'(x)<0$ for all $x>0$, $u(x)>c/b$ for all $x>0$ and $u(x)\to c/b$ as $x\to\infty$,}
  \eas
  or
  \bas
    & & \mbox{$u'(x)>0$ for all $x>0$, and either $u(x)<c/b$ for all $x>0$ and
    $u(x)\to c/b$ as $x\to\infty$,}\\
    & & \mbox{or $u$ is unbounded,}
  \eas
  or finally
  \bas
    & & \mbox{there exists a unique $x_0>0$ such that $u'(x_0)=0$, and we have $u''(x_0)>0$,} \\
    & & \mbox{$u(x)>u(x_0)>c/b$ for all $x\ne x_0$, $u'(x)<0$ for $x<x_0$, $u'(x)>0$ for $x>x_0$,} \\
    & & \mbox{and $u$ is unbounded.}
  \eas
  iii) \ If $b\le 0$ and $u\geq 0$ then $u'(x)>0$ for all $x>0$ and $u$ is unbounded.
\end{lem}
\proof
i) \ The substitution $x(t)=e^t, \tilde{u}(t)=u(x(t))$ transforms equation (\ref{0})
into
\be{udot}
  \tilde{u}'' = (a+1)\tilde{u}' + b\tilde{u} - c(e^{-t}\tilde{u}' - 1)^2.
\ee Being bounded and eventually monotonic, $\tilde{u}$ has a
limit $l$ as $t\to \infty$ and consequently
$\lim_{t\to\infty}\tilde{u}'(t)= 0$. From \eqref{udot} it now
follows that
$\lim_{t\to\infty}\tilde{u}''(t)= bl-c$. If $bl-c\ne 0$ then
$\lim_{t\to\infty}\tilde{u}''(t)\ne 0$ which is inconsistent with the
convergence of $\tilde{u}'$. This proves $bl-c=0$. For $b=0$ this is a
contradiction with $c>0$, for $b\ne 0$ it yields $l=c/b$.\abs
ii) \ If $u'(x)<0$
for all $x$ then $u$ is bounded and, by i), tends to $c/b$ as
$x\to\infty$ which is possible only if $u(x)>c/b$ for all $x$. If $u'(x)>0$ for all $x$
then thanks to monotonicity, $u(x)$ approaches a limit in $[0,\infty]$ as $x\to\infty$.
If this limit is finite it has to equal $c/b$ by virtue of i),
and in the remaining case $u$ is unbounded.\\
Suppose now there is $x_0>0$ such that $u'(x_0)=0$. Lemma \ref{one root} applied to equation \eqref{0} with $y\equiv u\ge 0$, $g(x,y):=by-c$ and $y^*:=c/b$ yields two alternatives,
the first of which is stated in part ii). The second alternative is not possible since it implies $0\le u\le c/b$ but at the same time $u'(x)<0$ for $x>x_0$ which means that
$u\not\to c/b$ as $x\to \infty$. A bounded solution not converging to $c/b$ contradicts part i).\abs
iii) \ If $u'$ is not positive everywhere then Lemma \ref{one root} applied to equation \eqref{0} with
$y\equiv u\ge 0$, $g(x,y):=by-c$ and $y^*:=\infty$ implies that there is $x_0$ such that
$u(x) < u(x_0)$ and $u'(x)<0$ for $x>x_0$.
Therefore $u$ is bounded and eventually monotonic. This contradicts i) when $b=0$. For $b<0$,
i) dictates that $u$ should converge to $c/b$ as $x\to\infty$, which contradicts $u\ge 0$ since $c/b<0$.
\qed
\begin{prop}\label{uprimeZero}
Suppose that $a+b>0$ and that $u(x)\leq x$ is a solution of
\eqref{0}, \eqref{0.1} on $(0,x_0)$ with some $x_0>0$. Then there
exists $x_1\in (0,x_0)$ such that $u'(x)>0$, $u''(x)<0$ and
$u'''(x)>0$ for all $x\in (0,x_1)$. Furthermore, in this case we
have \be{asympt}
    \lim_{x\to 0}\frac{u'(x)-1}{\sqrt{x}}=-\sqrt{\frac{a+b}{c}}.
\ee

\end{prop}
\proof  We recall that by  Lemma \ref{u1} $u'(0)=1$ and that by
the remark following Proposition 1.1, $u(x)$ is $C^\infty$ for
$x>0$ . As an immediate consequence we must have $u'(x)>0$ for all
sufficiently small $x>0$. On differentiating the equation
\eqref{0} we obtain \be{u3}
    x^2u'''+2xu''= axu''+(a+b)u'-2c(u'-1)u''
     \qquad \mbox{on } (0,x_0).
\ee Lemma \ref{one root} applied to equation \eqref{u3} with
$y\equiv u'$, $g(x,y):=(a+b)y$, $y^*:=0$ implies that $u''$ has a
constant non-zero sign near $x=0$.  This, together with $u(x)\le
x$ and $u'(0)=1$, yields that necessarily $u''(x)<0$ for all
sufficiently small $x>0$.\\ We now differentiate equation
\eqref{u3} once more to obtain
\be{u4}
    x^2u''''=((a-4)x-2c(u'-1))u'''+(2a+b-2)u''-2c(u'')^2
    \qquad \mbox{on } (0,x_0).
\ee
Lemma \ref{one root} applied to equation \eqref{u4} with $y\equiv u''\le 0$, $g(x,y):=(2a+b-2)y-2cy^2$ and
$y^*:=(2a+b-2)/(2c)$ implies that $u'''(x)$ has a constant non-zero sign near $x=0$.
Arguing by contradiction, we suppose that $u'''<0$ near $x=0$.
Since $u''<0$, this implies that $L:=\lim_{x\searrow 0}u''(x)$ exists and is finite.
This however contradicts equation \eqref{0}, since on integrating we find
$x^2u''(x)=Lx^2+o(x^2)$, $xu'(x)=x+Lx^2+o(x^2)$ and $u(x)=x+Lx^2/2+o(x^2)$ as $x\to 0$, and on substituting these
expressions into equation \eqref{0} one concludes that it cannot hold near $x=0$. We have thus proved $u'''>0$ near zero.\\
Next, dividing \eqref{0} by $x$ we obtain
\be{divided}
    xu''(x)=au'(x)+b\, \frac{u(x)}{x} + c\, \frac{(u'(x)-1)^2}{x}
    \qquad \mbox{for all } x\in (0,x_0).
\ee
Since $u''$ is increasing and negative, by \eqref{to 1} we find that
\bas
    u'(x)-1=\int_0^xu''(\xi)d\xi\le xu''(x)\le 0,
\eas
and, consequently,
\be{xu}
    xu''(x) \to 0 \qquad \mbox{as } x\to 0.
\ee
Substituting this into \eqref{divided} we obtain
\bas
    \lim_{x\to 0}c \, \frac{(u'(x)-1)^2}{x}=a+b.
\eas
Since $u'(x)-1\le 0$, this is equivalent to \eqref{asympt}.
\qed
\begin{cor}\label{nu} There is
a continuum of local solutions of (0.1),(0.2), with the property $0\leq u(x)\leq x$ and they all satisfy
\begin{align}
u(x) &= x-\frac{2}{3}\sqrt{\frac{a+b}{c}} x^{3/2} + o(x^{3/2})\label{aux1}\\
u'(x)&= 1-\sqrt{\frac{a+b}{c}} x^{1/2} + o(x^{1/2})\label{aux2}.
\end{align}
\end{cor}
\proof Multiplicity of solutions was proved in Proposition \ref{basicex}. Expansion \eqref{aux2} follows from \eqref{asympt}, and \eqref{aux1} follows by integration of \eqref{aux2}.
\qed
\begin{prop}\label{uprime}
  Let $u$ be a solution of \eqref{0}, \eqref{0.1} with $x_0=\infty$ such that $0\le
  u(x)\le x$ for all $x>0$.
  Then,  in addition to  \eqref{aux1} and \eqref{aux2}, we have that
  $u'(x)>0$, $u''(x)<0$ and $u'''(x)>0$ for all $x>0$. Moreover,
  \be{to 0}
    \lim_{x\to\infty} u'(x) = 0.
  \ee
\end{prop}
\proof
  The conclusion $u'(x)>0$ for all $x>0$ is a trivial consequence of
  Lemma \ref{monotony} iii) for $b\le 0$. In the case $b>0$, Lemma
  \ref{monotony} ii) implies that if $u$ is not increasing everywhere then there is $x_0>0$ such that $u'(x)<0$ on
  $(0,x_0)$ and this contradicts the facts that $u(0)=0$ and $u\ge0$.

  Next, Lemma \ref{one root} applied to equation \eqref{u3} with $y\equiv u'\ge 0$, $g(x,y):=(a+b)y\ge0$ and
  $y^*:=0$  shows that if $u''<0$ does not hold over $(0,\infty)$
  then $u''(x)>0$ for  all sufficiently large $x>0$. We show that the latter alternative is impossible.
  To this end, we let $v :=u'$ and $\tilde{v}(t):=v(x(t))$ with $x(t)=e^t$.
  Arguing by contradiction, since $v$ is eventually increasing and $u(x)\le x$, we must have $v\le 1$, which implies that
  $v$ and $\tilde{v}$ converge as $x\to\infty$ and $t\to\infty$, respectively. This in turn implies
  $v'(x(t))\to 0$ and $\tilde{v}'(t)\to 0$ as $t\to\infty$.  Since from \eqref{u3} we see that
  \be{vdot}
    \tilde{v}''(t)=(a-1)\tilde{v}'(t)+(a+b)\tilde{v}(t)-2c(\tilde{v}(t)-1)v'(x(t)),
  \ee
   we therefore obtain $\tilde{v}''(t) - (a+b)\tilde{v}(t)\to 0$ as $t\to \infty$. Since $\tilde{v}>0$ is increasing
  and convergent, so is $\tilde{v}''$. This is inconsistent with $\tilde{v}'(t)\to 0$ as $t\to\infty$.
  We have thus proved $u''(x)<0$ for all $x>0$.

  Now since $v$ is decreasing and bounded below by 0, it converges as
  $t\to\infty$. Therefore, both $v'(x(t))$ and $\tilde{v}'(t)$ converge to 0 as $t\to\infty$.
  From \eqref{vdot} we  thus obtain $\tilde{v}''(t) - (a+b)\tilde{v}(t)\to 0$ as $t\to \infty$.
  Should the limit of $\tilde{v}$ for $t\to\infty$ not be zero, the
  same would hold for $\tilde{v}''$. This, however, is in conflict
  with the convergence of $\tilde{v}'$ for $t\to\infty$. This proves statement \eqref{to 0}.

   Finally, in order to verify that $u'''>0$ throughout $(0,\infty)$, we note that from Lemma \ref{uprimeZero}
  we know that
  $u'''$ is positive near zero.
  Arguing by contradiction, we assume that $u''$ is not increasing everywhere.
  Lemma \ref{one root} applied to equation \eqref{u4} implies that in such case $u''$ must be eventually
  decreasing and therefore $\lim_{x\to \infty}u''(x)\neq 0$ which contradicts statement \eqref{to 0}.
\qed
The following inequality is related to the theory of speculative attacks, in which an agency artificially supports a (low) fixed price of an asset, using a limited amount of reserves. The inequality indicates that in order for the speculator to make expected profit or at least for her to break even, the price of the asset must always jump upwards after the speculative attack has exhausted the entire supply of the asset at the subsidized price.
\begin{cor}\label{inequality}
  Let $u$ be a solution of \eqref{0}, \eqref{0.1} with $a+b>0$ such that $0\le u(x)\le x$ for all $x>0$. Then
  \be{ineq}
    1+u'(x)>2 \, \frac{u(x)}{x} \qquad \mbox{for all } x>0.
  \ee
\end{cor}
\proof
  Since $u'''(x)>0$ for $x>0$ by Proposition \ref{uprime}, the
  function $u'$ is strictly convex on $[0,\infty)$. Therefore
  \bas
    \frac{u'(x)-u'(0)}{x} < u''(x) \qquad \mbox{for all } x>0.
  \eas
  On multiplying both sides by $x$, utilizing $u'(0)=1$ and integrating we obtain
  \bas
    u(x)-x<xu'(x)-u(x)
    \qquad \mbox{for all } x>0,
  \eas
  which yields the desired inequality.
\qed
\mysection{Uniqueness of the global solution bounded between \\ 0 and $x$}
\begin{prop}\label{uniqueness2}
There is one, and only one, solution $u$ of \eqref{0}, \eqref{0.1}
in $[0,\infty)$ which has the additional property that $0\le
u(x)\le x$ for all $x>0$. This solution necessarily satisfies
$u>0$, $u'>0$, $u''<0$, and $u'''>0$  on $(0,\infty)$.
\end{prop}
The proposition stems from the following result:

\begin{lem}\label{uniqueness1}
  Let $u\neq v$ be two solutions of \eqref{0} in $[0,\infty)$ which satisfy $u(0)=v(0)$, $u'(0)=v'(0)$
  and $u''\le 0$ on $(0,\infty)$. Then $w:=v-u$ satisfies either $w''>0$ on $(0,\infty)$ or $w''<0$ throughout
  $(0,\infty)$.
\end{lem}
\proof
  The function $w$ solves
  \bas
    x^2 w'' = axw' + bw - 2c(u'-1) w' - cw'^2
    \qquad \mbox{on } (0,\infty),
  \eas
  which on differentiation yields
  \be{w2}
    x^2 w''' = (a-2)xw'' + (a+b)w' - 2cu''w' - 2c(u'-1)w'' - 2cw'w''
    \qquad \mbox{on } (0,\infty).
  \ee
  Lemma \ref{one root} applied to equation \eqref{w2} with $y\equiv w'$, $g(x,y):=(a+b-2cu''(x))y$ and $y^*:=0$
  shows that $w''$ can have at most one root.
  Now the existence of such a root
  $x_0>0$ would imply either $w'>0$, $w''(x)<0$ for $0<x<x_0$ and $w''(x)>0$ for $x>x_0$; or
  $w'<0$, $w''(x)>0$ for $0<x<x_0$ and $w''(x)<0$ for $x>x_0$. This however contradicts $w'(0)=0$.
  Thus one must have either $w''(x)>0$ or $w''(x)<0$ for all $x>0$.
\qed
\proofc of Proposition \ref{uniqueness2}.\quad
  Suppose there are two solutions $u\neq v$ bounded between 0 and $x$ and let $w:=v-u$.
  Proposition \ref{uprime} yields $\lim_{x\to\infty} u'(x)=\lim_{x\to\infty}v'(x) =0$ and therefore
  \be{wprime}
    \lim_{x\to\infty} w'(x)=0.
  \ee
  On the other hand, Proposition \ref{uprime} also gives $u'(0)=v'(0)=1$, implying $w'(0)=0$. We can thus employ
  Lemma \ref{uniqueness1}
  to obtain that either $w''(x)>0$ or  $w''(x)<0$ for all $x>0$. In view of $w'(0)=0$ both alternatives contradict
  \eqref{wprime}.
  The claimed
  further properties of the unique solution bounded between $0$ and $x$ follow from Proposition \ref{uprime}.
\qed
\mysection{Finer aspects of non-uniqueness}
By Corollary \ref{asympt} there is a continuum of local solutions
sharing the same asymptotics up to degree 3/2 at 0. The solutions
are parametrized by their values at any $x>0$.  In this section we
establish existence of disjoint continua of local solutions of
(0.1), (0.2) distinguished by refinement of their asymtotics at 0.

We first show  that the higher order terms specification of
\eqref{aux2}  can be somewhat sharpened. Denote
$$k=\sqrt{\frac{a+b}{c}}$$
\begin{lem}\label{lem_n1}
  Let $u, x_0, x_1$ be as in Corollary \ref{asympt}.
  Then there exists $l>0$ such that
  \be{kl}
    \Big| \frac{u'(x)-1}{\sqrt{x}} + k \Big| \le l\sqrt{x}  \qquad \mbox{for all } x\in (0,x_1).
  \ee
If $u$ is the unique solution provided by Proposition
\ref{uniqueness2} then \eqref{kl} holds for all $0\le x<\infty$
for some $l>0$.
\end{lem}
\proof
  Since $\frac{u'(x)-1}{\sqrt{x}}\to -k$ as $x\searrow 0$ by Proposition \ref{uprimeZero},
  we can find $x_1>0, \ c_1>0$ such that
  \be{n1.2}
    |u'(x)-1| \le c_1 \sqrt{x}
    \qquad \mbox{for } 0< x<x_1.
  \ee
  Therefore trivially
  \be{n1.3}
    |au'(x)-a| \le c_1 |a| \sqrt{x}
    \qquad \mbox{for all } 0<x<x_1
  \ee
  and also
  \be{n1.4}
    \Big| \frac{bu(x)}{x}-b \Big|
    = \frac{|b|}{x} \cdot \Big| \int_0^x (u'(y)-1) dy \Big|
    \le \frac{c_1 |b|}{x} \cdot \int_0^x \sqrt{y}{dy}
    = \frac{2c_1 |b|}{3} \sqrt{x}
    \qquad \mbox{for all } 0<x<x_1.
  \ee
  Taking $x_1$ sufficiently small,  using Proposition \ref{uprimeZero} we obtain that
  $u''\le 0$ and $u'''\ge 0$ and hence $|xu''(x)| \le |u'(x)-1|$ for all $0<x<x_1$, from (\ref{n1.2})
  we also infer that
  \bas
    |xu''(x)| \le c_1 \sqrt{x}
    \qquad \mbox{for all } 0<x<x_1.
  \eas
  Accordingly, from (\ref{0}) we obtain that $\varphi(x):=\frac{u'(x)-1}{\sqrt{x}}, 0<x<x_1$, satisfies
  \bea{n1.44}
    \Big| c\varphi^2(x)-ck^2\Big|
    &=& \bigg| (au'(x)-a) + \Big(\frac{bu(x)}{x}-b\Big) - xu''(x) \bigg| \nn\\
    &\le& c_2\sqrt{x}
    \qquad \mbox{for all } 0<x<x_1
  \eea
  with $c_2:=c_1 |a| + \frac{2c_1|b|}{3}+c_1$. Therefore $\varphi^2(x) \le k^2+ \frac{c_2}{c}\sqrt{x}$, so that
  \be{n1.5}
    -\varphi(x) = |\varphi(x)| \le k \cdot \sqrt{1+\frac{c_2\sqrt{x}}{ck^2}}
    \le k \cdot \Big(1+\frac{c_2\sqrt{x}}{2ck^2}\Big)
    \qquad \mbox{for all } 0<x<x_1,
  \ee
  where we have used that $\sqrt{1+\xi} \le 1+\frac{\xi}{2}$ for all $\xi>0$.
  Likewise, (\ref{n1.44}) entails that $\varphi^2(x) \ge k^2 - \frac{c_2}{c} \sqrt{x}$ for all $x>0$. Thus, since
  $\sqrt{1-\xi} \ge 1-\frac{\xi}{\sqrt{2}}$ for all $\xi \in (0,\frac{1}{2})$, we see that
  \be{n1.6}
    -\varphi(x) \ge k\cdot\sqrt{1-\frac{c_2\sqrt{x}}{ck^2}}
    \ge k-\frac{c_2\sqrt{x}}{\sqrt{2}ck}
    \qquad \mbox{for all  }x\in (0,x_2),
  \ee
  where $x_2:=\min\{x_1,(\frac{ck^2}{2c_2})^2\} $.

If $u$ is the global solution of Proposition \ref{uniqueness2},
then  (\ref{kl})
  is obvious in $[x_1,\infty)$ because
  $u'(x)\to 0$ as $x\to\infty$. This, together with (\ref{n1.5}) and (\ref{n1.6})
  concludes  the proof of the lemma. \qed
The core of our approach will be formed by the usage of on ordered pair of
sub- and supersolutions which deviate from our original solution by an exponentially small term near $x=0$.
As a preparation, let us compute the action of the operator $\E$, as defined in (\ref{8.2}), on such functions.
\begin{lem}\label{lem_n2}
  Suppose that $u$ is a solution of (\ref{0}), (\ref{0.1}) in $(0,x_0)$ for some $x_0>0$. Then for all $\eps>0$,
  $\alpha\in\R$ and $\beta\in\R$, the function $v$ defined by
  \be{n2.1}
    v(x):=u(x)-\eps x^\alpha e^{-\frac{\beta}{\sqrt{x}}}, \qquad x\in (0,x_0),
  \ee
  satisfies
  \bea{n2.2}
    \E v &=& \eps e^{-\frac{\beta}{\sqrt{x}}} \cdot \Bigg\{
    \frac{\beta^2}{4} x^{\alpha-1} \nn\\
    & & \hspace*{18mm}
    + \Big[ \Big(\alpha-\frac{3}{2}\Big) \frac{\beta}{2} + \frac{\alpha\beta}{2} - \frac{a\beta}{2}\Big] \cdot
        x^{\alpha-\frac{1}{2}} \nn\\
    & & \hspace*{18mm}
    + \Big[\alpha(\alpha-1) - a\alpha - b \Big] \cdot x^\alpha \nn\\
    & & \hspace*{18mm}
    +2c \cdot \frac{u'-1}{\sqrt{x}} \cdot \Big[ \frac{\beta}{2} x^{\alpha-1} + \alpha x^{\alpha-\frac{1}{2}}\Big] \nn\\
    & & \hspace*{18mm}
    -\eps c \cdot \Big[ \frac{\beta}{2} x^{\alpha-\frac{3}{2}} + \alpha x^{\alpha-1}\Big]^2 \cdot
    e^{-\frac{\beta}{\sqrt{x}}} \Bigg\}
  \eea
  for all $x\in (0,x_0)$.
\end{lem}
\proof
  We write
  \bas
    w(x):=\eps x^\alpha e^{-\frac{\beta}{\sqrt{x}}}, \qquad x\in (0,x_0),
  \eas
  so that $v=u-w$ and consequently
  \bas
    \E v &=& - x^2 u'' + axu' +bu \\
    & & + x^2 w'' -axw' -bw \\
    & & +c (u'-w'-1)^2
    \qquad \mbox{in } (0,x_0).
  \eas
  Since $(u'-w'-1)^2 = (u'-1)^2 - 2(u'-1)w' + w'^2$, using that $\E u \equiv 0$ we see that
  \be{n2.3}
    \E v = x^2 w'' - axw' -bw
    +2c(u'-1) w' - cw'^2
    \qquad \mbox{in } (0,x_0).
  \ee
  We now compute
  \bas
    w'(x)=\eps \frac{\beta}{2} x^{\alpha-\frac{3}{2}} e^{-\frac{\beta}{\sqrt{x}}}
    +\eps \alpha x^{\alpha-1} e^{-\frac{\beta}{\sqrt{x}}}
  \eas
  and
  \bas
    w''(x) = \eps \frac{\beta^2}{4} x^{\alpha-3} e^{-\frac{\beta}{\sqrt{x}}}
    + \eps \Big(\alpha-\frac{3}{2} \Big) \frac{\beta}{2} x^{\alpha-\frac{5}{2}} e^{-\frac{\beta}{\sqrt{x}}}
    + \eps \frac{\alpha\beta}{2} x^{\alpha-\frac{5}{2}} e^{-\frac{\beta}{\sqrt{x}}}
    + \eps \alpha(\alpha-1) x^{\alpha-2} e^{-\frac{\beta}{\sqrt{x}}}
  \eas
  and hence obtain from (\ref{n2.3}) that
  \bas
    \E v &=&
    \eps e^{-\frac{\beta}{\sqrt{x}}} \cdot \Bigg\{
    \frac{\beta^2}{4} x^{\alpha-1}
    + \Big(\alpha-\frac{3}{2} \Big) \frac{\beta}{2} x^{\alpha-\frac{1}{2}}
    +\frac{\alpha\beta}{2} x^{\alpha-\frac{1}{2}}
    + \alpha(\alpha-1) x^\alpha \\
    & & \hspace*{18mm}
    -\frac{a\beta}{2} x^{\alpha-\frac{1}{2}} - a\alpha x^\alpha - bx^\alpha \\
    & & \hspace*{18mm}
    +2c\cdot \frac{u'-1}{\sqrt{x}} \cdot \sqrt{x} \cdot
        \Big[ \frac{\beta}{2} x^{\alpha-\frac{3}{2}} + \alpha x^{\alpha-1} \Big] \\
    & & \hspace*{18mm}
    -\eps c \Big[ \frac{\beta}{2} x^{\alpha-\frac{3}{2}} + \alpha x^{\alpha-1} \Big]^2 e^{-\frac{\beta}{\sqrt{x}}} \Bigg\}
    \qquad \mbox{in } (0,x_0).
  \eas
  On straightforward rearrangements, this yields (\ref{n2.2}).
\qed
We can now identify an appropriate family of subsolutions for (\ref{0}).
\begin{lem}\label{lem_n3}
  Suppose that for some $x_0>0$, $u$ is a solution of (\ref{0}), (\ref{0.1}) in $(0,x_0)$ satisfying (\ref{kl})
  with some $l>0$.
  Then for any real number $\alpha<\frac{3}{2}+a-2cl$ there exists $x_3=x_3(\alpha) \in (0,x_0)$ such that for each
  $\eps>0$, the function $\uv$ defined by
  \be{n3.1}
    \uv(x):=u(x)-\eps x^\alpha e^{-\frac{4ck}{\sqrt{x}}}, \qquad x\in (0,x_3),
  \ee
  satisfies
  \be{n3.2}
    \E \uv < 0 \qquad \mbox{in } (0,x_3).
  \ee
\end{lem}
\proof
  Given $\alpha<\frac{3}{2}+a-2cl$, thanks to the positivity of $c$ and $k$ we can find $x_1\in (0,x_0)$ such that both
  \be{n3.3}
    \frac{2ck}{\sqrt{x_1}} + \alpha \ge 0
  \ee
  and
  \be{n3.4}
    \frac{2ck\cdot (\frac{3}{2}+a-2cl-\alpha)}{\sqrt{x_1}} > \alpha(\alpha-1)-a\alpha-b+2cl\alpha
  \ee
  hold. Then for $\eps>0$ we define $\uv$ as in (\ref{n3.1}) and thus obtain from Lemma \ref{lem_n2}, applied to
  $\beta:=4ck$, that (\ref{n2.2}) holds for $v=\uv$ and all $x\in (0,x_1)$.
  In order to utilize (\ref{kl}) appropriately, we note that in view of (\ref{n3.3}) we have
  \bas
    \frac{\beta}{2} x^{\alpha-1} + \alpha x^{\alpha-\frac{1}{2}}
    = \Big(\frac{2ck}{\sqrt{x}} + \alpha \Big) \cdot x^{\alpha-\frac{1}{2}} \ge 0
    \qquad \mbox{for all } x\in (0,x_1).
  \eas
  We may thus multiply this by the inequality
  \bas
    \frac{u'(x)-1}{\sqrt{x}} \le -k + l\sqrt{x}
    \qquad \mbox{for all } x\in (0,x_1),
  \eas
  as resulting from (\ref{kl}), to infer from (\ref{n2.2}) on dropping a nonpositive term that
  \bas
    \E \uv &\le& \eps e^{-\frac{\beta}{\sqrt{x}}} \cdot \Bigg\{
    \frac{\beta^2}{4} x^{\alpha-1} \nn\\
    & & \hspace*{18mm}
    + \Big[ \Big(\alpha-\frac{3}{2}\Big) \frac{\beta}{2} + \frac{\alpha\beta}{2} - \frac{a\beta}{2}\Big] \cdot
        x^{\alpha-\frac{1}{2}} \nn\\
    & & \hspace*{18mm}
    + \Big[\alpha(\alpha-1) - a\alpha - b \Big] \cdot x^\alpha \nn\\
    & & \hspace*{18mm}
    +2c (-k+l\sqrt{x}) \cdot \Big[ \frac{\beta}{2} x^{\alpha-1} + \alpha x^{\alpha-\frac{1}{2}}\Big]
    \Bigg\} \\[1mm]
    &=& \eps e^{-\frac{\beta}{\sqrt{x}}} \cdot \Bigg\{
    \Big[\frac{\beta^2}{4} -ck\beta \Big] \cdot x^{\alpha-1} \nn\\
    & & \hspace*{18mm}
    + \Big[ \Big(\alpha-\frac{3}{2}\Big) \frac{\beta}{2} + \frac{\alpha\beta}{2} - \frac{a\beta}{2}
        -2ck\alpha+cl\beta \Big] \cdot x^{\alpha-\frac{1}{2}} \nn\\
    & & \hspace*{18mm}
    + \Big[\alpha(\alpha-1) - a\alpha - b +2cl\alpha \Big] \cdot x^\alpha \Bigg\}
    \qquad \mbox{in } (0,x_1).
  \eas
  Here the leading term disappears because $\beta=4ck$, whereas
  \bas
    \Big(\alpha-\frac{3}{2}\Big) \frac{\beta}{2} + \frac{\alpha\beta}{2} - \frac{a\beta}{2} -2ck\alpha+cl\beta
    &=& 2ck \cdot \Big[ \alpha-\frac{3}{2} + \alpha-a-\alpha+2cl\Big] \\
    &=& 2ck \cdot \Big[\alpha-\frac{3}{2}-a+2cl\Big] <0.
  \eas
  Hence, using (\ref{n3.4}) we infer that
  \bas
    \E \uv &\le& \eps e^{-\frac{\beta}{\sqrt{x}}} \cdot x^\alpha \cdot \bigg\{
    -\frac{2ck \cdot (\frac{3}{2}+a-2cl-\alpha)}{\sqrt{x}}
    +\alpha(\alpha-1) -a\alpha - b + 2cl\alpha \bigg\} \\[1mm]
    &<& 0
    \qquad \mbox{in } (0,x_1),
  \eas
  as desired.
\qed
Our construction of supersolutions to (\ref{0}) is similar.
\begin{lem}\label{lem_n4}
  Let $u$ be a solution of (\ref{0}), (\ref{0.1}) in $(0,x_0)$ for some $x_0>0$, which satisfies (\ref{kl})
  with some $l>0$.
  Then for each $\eps_0>0$ and $\alpha>\frac{3}{2}+a+2cl$ one can pick $x_4=x_4(\alpha) \in (0,x_0)$ such that for any
  $\eps \in (0,\eps_0)$, the function $\ov$ defined by
  \be{n4.1}
    \ov(x):=u(x)-\eps x^\alpha e^{-\frac{4ck}{\sqrt{x}}}, \qquad x\in (0,x_4),
  \ee
  satisfies
  \be{n4.2}
    \E \ov > 0 \qquad \mbox{in } (0,x_4).
  \ee
\end{lem}
\proof
  Relying on our assumption on $\alpha$, let us first pick $x_4 \in (0,x_0)$ such that besides again
  \be{n4.3}
    \frac{2ck}{\sqrt{x_3}} + \alpha \ge 0,
  \ee
  the inequality
  \be{n4.4}
    \frac{2ck(\alpha-\frac{3}{2}-a-2cl)}{\sqrt{x_4}}
    > \big| \alpha(\alpha-1) - a\alpha - b - 2cl\alpha \big|
    + c\cdot \sup_{x\in (0,x_4)}
    \bigg\{ x^{-\alpha} \cdot \Big[ 2ck x^{\alpha-\frac{3}{2}} + \alpha x^{\alpha-1}\Big]^2
        \cdot e^{-\frac{4ck}{\sqrt{x}}} \bigg\}
  \ee
  holds.
  Then writing $\beta:=4ck$, we again have $\frac{\beta}{2} x^{\alpha-1}+\alpha x^{\alpha-\frac{1}{2}}\ge 0$
  in $(0,x_4)$. Hence, using that (\ref{kl}) implies that
  \bas
    \frac{u'(x)-1}{\sqrt{x}} \ge -k-l\sqrt{x}
    \qquad \mbox{for all } x\in (0,x_0),
  \eas
  from (\ref{n2.2}) we obtain
  \bas
    \E \ov &\ge& \eps e^{-\frac{4ck}{\sqrt{\beta}}} \cdot \Bigg\{
    \frac{\beta^2}{4} x^{\alpha-1} \nn\\
    & & \hspace*{18mm}
    + \Big[ \Big(\alpha-\frac{3}{2}\Big) \frac{\beta}{2} + \frac{\alpha\beta}{2} - \frac{a\beta}{2}\Big] \cdot
        x^{\alpha-\frac{1}{2}} \nn\\
    & & \hspace*{18mm}
    + \Big[\alpha(\alpha-1) - a\alpha - b \Big] \cdot x^\alpha \nn\\
    & & \hspace*{18mm}
    +2c \cdot (-k-l\sqrt{x}) \cdot \Big[ \frac{\beta}{2} x^{\alpha-1} + \alpha x^{\alpha-\frac{1}{2}}\Big] \nn\\
    & & \hspace*{18mm}
    -\eps_0 c \cdot \Big[ \frac{\beta}{2} x^{\alpha-\frac{3}{2}} + \alpha x^{\alpha-1}\Big]^2 \cdot
    e^{-\frac{\beta}{\sqrt{x}}} \Bigg\}
    \qquad \mbox{in } (0,x_4).
  \eas
  Since $\beta=4ck$, this reduces to
  \bas
    \E \ov \ge \eps e^{-\frac{4ck}{\sqrt{x}}} \cdot x^\alpha \cdot \Bigg\{
    \frac{2ck(\alpha-\frac{3}{2}-a-2cl)}{\sqrt{x}} + \alpha(\alpha-1)-a\alpha-b-2cl\alpha \\
    -\eps_0 c x^{-\alpha} \cdot \Big[ 2ck x^{\alpha-\frac{3}{2}} + \alpha x^{\alpha-1}\Big]^2
        \cdot e^{-\frac{4ck}{\sqrt{x}}} \Bigg\}
    \qquad \mbox{in } (0,x_4),
  \eas
  and hence (\ref{n4.4}) asserts (\ref{n4.2}).
\qed
By means of Proposition \ref{prop1a}, we can now infer the
existence of  infinitely many classes of continua of local
solutions to (\ref{0}), (\ref{0.1}). Here we shall strongly rely
on the fact that the numbers $\alpha$ in Lemma \ref{lem_n3} and
Lemma \ref{lem_n4} can be chosen in such a way that the functions
$\uv$ and $\ov$ defined in (\ref{n3.1}) and (\ref{n4.1}) are
ordered appropriately.
\begin{prop}\label{fine}
  Suppose that for some $x_0>0$, $u$ is a solution of (\ref{0}), (\ref{0.1}) in $(0,x_0)$ satisfying
  (\ref{kl}) with some $l>0$.
  Then for all $\ualpha \in (-\infty,\frac{3}{2}+a-2cl)$ and $\oalpha \in (\frac{3}{2}+a+2cl,\infty)$
  there exists $\hxz \in (0,x_0)$ such that given  $\eps \in (0,1)$ and any $\hu_0$ satisfying
  \bas
    \hu_0 \in  \Big( u(x)-\eps x^{\ualpha} e^{-\frac{4ck}{\sqrt{x}}} \, , \,
    u(x)-\eps x^{\oalpha} e^{-\frac{4ck}{\sqrt{x}}} \Big) \qquad \mbox{at } x=\hxz,
  \eas
  (\ref{0}), (\ref{0.1}) possesses a solution $\hu$ on $(0,\hxz)$ fulfilling
  \be{n5.1}
    u(x)-\eps x^{\ualpha} e^{-\frac{4ck}{\sqrt{x}}}
    \le \hu(x) \le
    u(x)-\eps x^{\oalpha} e^{-\frac{4ck}{\sqrt{x}}}
    \qquad \mbox{for all } x\in (0,\hxz)
  \ee
 as well  as $\hu(\hxz)=\hu_0$.
\end{prop}
\proof
  We invoke Lemma \ref{lem_n3} and Lemma \ref{lem_n4} with $\eps_0:=1$ to obtain $x_3\in (0,x_0)$ and $x_4\in (0,x_0)$
  such that for any $\eps \in (0,1)$, the functions $\uv$ and $\ov$ defined by
  $\uv(x):=u(x)-\eps x^{\ualpha} e^{-\frac{4ck}{\sqrt{x}}}, x\in (0,x_3)$, and
  $\ov(x):=u(x)-\eps x^{\oalpha} e^{-\frac{4ck}{\sqrt{x}}}, x\in (0,x_4)$, have the properties
  $\E \uv < 0$ on $(0,x_3)$ and $\E \ov > 0$ on $(0,x_4)$.
  Writing $\hxz:=\min\{1,x_3,x_4\}$, we moreover see using $\ualpha<\oalpha$ that $\uv<\ov$ throughout $(0,\hxz)$.
  Therefore the claim results upon an application of Proposition \ref{prop1a}(i).
\qed
\textbf{Acknowledgments.} We would like to thank an anonymous
referee for very detailed and thoughtful comments. This research
was undertaken while Ale\v s \v{C}ern\'y visited Comenius
University in Bratislava in spring 2012. The financial support of
the V\'{U}B Foundation under the "Visiting professor 2011" funding
scheme is gratefully acknowledged. Pavol Brunovsk\'y gratefully
acknowledges support by the grants VEGA 1/0711/12 and VEGA
1/2429/12.


\begin{thebibliography}{[5]}

\bibitem{Barles}
\sc Barles, G.: {\it "Convergence of numerical schemes for degenerate parabolic equations arising in finance theory"}. \rm In L. C. G. Rogers and D. Talay eds., Numerical methods in finance, 1-21 (1997)
%
\bibitem{bernstein} Bernstein, C.: {\it Sur certaines équations
différentielles ordinaires du second ordre}. \rm C. R. Acad. Sci.
Paris 138(1904), 950-951
%
\bibitem{cerny}
\sc \v{C}ern\'{y}, A.: {\it "Currency Crises: Introduction of Spot Speculators"}.
\rm Int. J. Financ. Econ. {\bf 4}, 75-89 (1999)
%
\bibitem{Cidetal}
\sc Cid, J. \'{A}., L\'{o}pez-Pouso, O., and R. L\'{o}pez-Pouso: {\it "Existence of infinitely many solutions for second-order singular initial value problems
with an application to nonlinear massive gravity"}. \rm Nonlinear Anal. Real World Appl. {\bf 12}, 2596-2606 (2011)
%
\bibitem{decoster}
{\sc De Coster, C., and P. Habets: {\it Two-point
boundary value problems: Upper and lower solutions}. \rm
Mathematics in Science and Engineering 205, Elsevier 2006}
%
\bibitem{hartman} Hartman, Philip: {\it Oridnary Differential
Equations}. \rm Wiley 1964
%
\bibitem{liang} \sc Liang, J.: \it "A singular initial value problem and self-similar solutions of a nonlinear dissipative wave equation" \rm J. Differ. Equ., {\bf 246}, 819-844 (2009)
%
\bibitem{nagumo} \sc Nagumo, M.: {\it "\"{U}ber die
Differentialgleichung $y''=f(x,y,y')$"} \rm Proc. Phys. Math. Soc.
Japan {\bf 19}, 861-866 (1937)
%
%
\bibitem{wirl}
\sc Wirl, F.: {\it "Energy prices and carbon taxes under
uncertainty about global warming"}. \rm Env. Resour. Econ.
{\bf 36}, 313-340 (2007)

\end{thebibliography}
\end{document}